\documentclass[12pt]{amsart}
\usepackage{amsmath, amssymb, fullpage, stmaryrd, amsthm}
\usepackage[all]{xy}

\numberwithin{equation}{section}

\newtheorem{theorem}[equation]{Theorem}
\newtheorem{lemma}[equation]{Lemma}
\newtheorem{prop}[equation]{Proposition}
\newtheorem{cor}[equation]{Corollary}
\theoremstyle{definition}

\def\CC{\mathbb{C}}
\def\FF{\mathbb{F}}
\def\QQ{\mathbb{Q}}

\def\ZZ{\mathbb{Z}}
\def\frako{\mathfrak{o}}
\def\frakp{\mathfrak{p}}
\def\bt#1{\tilde{\mathbf{#1}}}
\def\calR{\mathcal{R}}

\def\bB{\mathbf{B}}

\def\barF{\overline{F}}
\def\barx{\overline{x}}
\def\bartheta{\overline{\theta}}
\def\barQQ{\overline{\QQ}}
\def\barZZ{\overline{\ZZ}}

\def\Betti{\mathrm{Betti}}
\def\crys{\mathrm{crys}}
\def\dR{\mathrm{dR}}
\def\et{\mathrm{et}}
\def\rig{\mathrm{rig}}
\def\st{\mathrm{st}}
\def\tors{\mathrm{tors}}

\def\dual{\vee}

\DeclareMathOperator{\Ext}{Ext}
\DeclareMathOperator{\Fil}{Fil}
\DeclareMathOperator{\Frac}{Frac}
\DeclareMathOperator{\Gal}{Gal}
\DeclareMathOperator{\GL}{GL}
\DeclareMathOperator{\Hom}{Hom}
\DeclareMathOperator{\id}{id}

\DeclareMathOperator{\rank}{rank}

\begin{document}

\title{Some new directions in $p$-adic Hodge theory}
\author{Kiran S. Kedlaya}
\date{February 3, 2009}

\maketitle

\begin{abstract}
We recall some basic constructions from $p$-adic Hodge theory, then
describe some recent results in the subject. We chiefly discuss the
notion of $B$-pairs, introduced recently by Berger, which provides
a natural enlargement of the category of $p$-adic Galois representations.
(This enlargement, in a different form, figures in recent work
of Colmez, Bella\"\i che, and Chenevier on trianguline representations.)
We also discuss results of Liu that indicate that the formalism of
Galois cohomology, including Tate local duality, extends to $B$-pairs.
\end{abstract}

\section{Setup and overview}

Throughout, $K$ will denote a finite extension of the field $\QQ_p$
of $p$-adic numbers, and $G_K = \Gal(\barQQ_p/K)$ will denote the
absolute Galois group of $K$. We will write $\CC_p$ for the completion
of $\barQQ_p$; it is algebraically closed, and complete for a
nondiscrete valuation. For any field $F$
carrying a valuation (like $K$ or $\CC_p$), we write $\frako_F$ for the
valuation subring.

One may think of $p$-adic Hodge theory as the $p$-adic analytic study of
\emph{$p$-adic representations} of $G_K$, by which we mean
finite dimensional $\QQ_p$-vector spaces $V$ equipped with
continuous homomorphisms $\rho: G_K \to \GL(V)$. (One might want to allow
$V$ to be a vector space over a finite extension of $\QQ_p$; for ease
of exposition, I will only retain the $\QQ_p$-structure in this discussion.)
A typical example of a $p$-adic representation is the (geometric) $p$-adic
\'etale cohomology $H^i_{\et}(X_{\barQQ_p}, \QQ_p)$ of an
algebraic variety $X$ defined over $K$.
Another typical example is the restriction to $G_K$ of a global
Galois representation $G_F \to \GL(V)$, where $F$ is a number field,
$K$ is identified with the completion of $F$ at a place above $p$,
and $G_K$ is identified with a subgroup of $G_F$;
this agrees with the previous construction if the global representation
itself arises as $H^i_{\et}(X_{\barF}, \QQ_p)$ for a variety $X$ over $F$.

Examples of this sort may be thought of as having a ``geometric origin'';
it turns out that there are many $p$-adic representations without this
property. For instance, there are several constructions that start
with the $p$-adic representations associated to classical modular forms
(which do have a geometric origin), and produce new $p$-adic representations
by $p$-adic interpolation. These constructions include the $p$-adic
families of Hida \cite{hida}, and the eigencurve of Coleman and Mazur
\cite{coleman-mazur}. (Note that these are \emph{global} representations,
so one has to first restrict to a decomposition group to view them within
our framework.)

Our purpose here is to first recall the basic framework of
$p$-adic Hodge theory, then describe some new results. One important
area of application is Colmez's
work on the $p$-adic local Langlands correspondence for
2-dimensional representations of $G_{\QQ_p}$
\cite{colmez, colmez2, colmez3, colmez4}.

\section{The basic strategy}

The basic methodology of $p$-adic Hodge theory, as introduced by Fontaine,
is to linearize the data of a $p$-adic representation $V$
by tensoring with a suitable topological $\QQ_p$-algebra $\bB$
equipped with a continuous $G_K$-action, then forming the space 
\[
D_{\bB}(V) = (V \otimes_{\QQ_p} \bB)^{G_K}
\]
of Galois invariants. 
One usually asks for $\bB$ to be \emph{regular},
which means that $\bB$ is a domain,
$(\Frac \bB)^{G_K} = \bB^{G_K}$ (so the latter is a field),
and any $b \in \bB$ for which $\QQ_p \cdot
b$ is stable under
$G_K$ satisfies $b \in \bB^\times$. It also forces the map
\begin{equation} \label{eq:admissible}
D_{\bB}(V) \otimes_{\bB^{G_K}} \bB \to V \otimes_{\QQ_p} \bB
\end{equation}
to be an injection; one says that $V$ is 
\emph{$\bB$-admissible} if \eqref{eq:admissible} is a bijection,
or equivalently, if the inequality $\dim_{\bB^{G_K}} D_{\bB}(V) 
\leq \dim_{\QQ_p} V$ is an equality.

In particular, Fontaine defines period rings
$\bB_{\crys}, \bB_{\st}, \bB_{\dR}$; we say $V$ is
\emph{crystalline}, \emph{semistable}, or \emph{de Rham} if it is
admissible for the corresponding period ring.
We will define these rings shortly; for now,
consider the following result, conjectured by Fontaine and Jannsen, and
proved by Fontaine-Messing, Faltings, Tsuji, et al.
\begin{theorem}
Let $V = H^i_{\et}(X_{\barQQ_p}, \QQ_p)$ for $X$ a smooth proper
scheme over $K$.
\begin{enumerate}
\item[(a)]
The representation $V$ is de Rham,
and there is a canonical isomorphism
of filtered $K$-vector spaces
\[
D_{\dR}(V) \cong H^i_{\dR}(X, K).
\]
\item[(b)]
If $X$ extends to a smooth proper $\frako_K$-scheme, then
$V$ is crystalline.
\item[(c)]
If $X$ extends to a semistable $\frako_K$-scheme, then
$V$ is semistable.
\end{enumerate}
\end{theorem}

In line with the previous statement, the following result was
conjectured by Fontaine; its proof combines a result of
Berger with a theorem concerning
$p$-adic differential equations due to Andr\'e, Mebkhout, and
the author.
\begin{theorem} \label{T:pst}
Let $V$ be a de Rham representation of $G_K$. Then $V$ is
potentially semistable; that is, there exists a finite
extension $L$ of $K$ such that the restriction of $V$ to $G_L$ is 
semistable.
\end{theorem}

\section{Period rings}

We now describe some of the key rings in Fontaine's theory. 
(We recommend \cite{berger-dwork} for a more detailed introduction.)
Keep in mind
that everything we write down will carry an
action of $G_{\QQ_p}$, so we won't say this explicitly each time.

Write $\frako$ for $\frako_{\CC_p}$.
The ring $\frako/p\frako$
admits a $p$-power Frobenius map; let $\bt{E}^+$ be the inverse limit
of $\frako/p \frako$ under Frobenius. That is, an element of 
$\bt{E}^+$ is a sequence
$\barx = (\barx_n)_{n=0}^\infty$ of elements of $\frako/p\frako$ with 
$\barx_n = \barx_{n+1}^p$.
By construction, Frobenius is a bijection on $\bt{E}^+$.

If $\barx \in \bt{E}^+$ is nonzero, then $p^n v_{\CC_p}(\barx_n)$
is constant for those $n$ for which $\barx_n \neq 0$. The resulting function
\[
v_E(\barx) = \lim_{n \to \infty} p^n v_{\CC_p}(\barx_n)
\]
is a valuation, and $\bt{E}^+$ is a valuation ring for this valuation;
in particular, $\bt{E}^+$ is an integral domain (even though
$\frako/p\frako$ is nonreduced). It will be convenient to fix the choice
of an element $\epsilon \in \bt{E}^+$ such that $\epsilon_n$ is a primitive
$p^n$-th root of unity; then $\bt{E}^+$ is the valuation ring in
a completed algebraic closure of $\FF_p((\epsilon - 1))$.

Given $\barx \in \bt{E}^+$, let $x_n \in \frako$ be any lift of $\barx_n$.
The elementary fact that
\[
a \equiv b \pmod{p^m} \implies a^p \equiv b^p \pmod{p^{m+1}}
\]
implies first that the sequence $(x_n^{p^n})_{n=0}^\infty$
is $p$-adically convergent,
and second that the limit is independent of the choice of the $x_n$.
We call this limit $\bartheta(\barx)$; the resulting function
$\bartheta: \bt{E}^+ \to \frako$ is not a homomorphism,
but it is multiplicative.
In particular, $\bartheta(\epsilon) = 1$.

Let $\bt{A}^+ = W(\bt{E}^+)$ be the ring of $p$-typical Witt vectors with
coefficients in $\bt{E}^+$. Although this ring is non-noetherian (because
the valuation on $\bt{E}^+$ is not discrete), one should still think of it
as a ``two-dimensional'' local ring, since $\bt{A}^+/p\bt{A}^+ \cong \bt{E}^+$.
By properties of Witt vectors, the multiplicative map $\bartheta$
lifts to a ring homomorphism
$\theta: \bt{A}^+ \to \frako$ taking a Teichm\"uller lift
$[\barx]$ to $\bartheta(\barx)$. (The Teichm\"uller lift $[\barx]$ is the unique
lift of $\barx$ which has $p^n$-th roots for all $n \geq 0$.)
Also, there is a Frobenius map $\phi: \bt{A}^+ \to \bt{A}^+$ lifting 
the usual Frobenius on $\bt{E}^+$.
Put $\bt{B}^+ = \bt{A}^+[\frac 1p]$; then $\theta$ extends to a map
$\theta: \bt{B}^+ \to \CC_p$, and $\phi$ also extends.

With this, we are ready to make Fontaine's rings. 
It can be shown that in $\bt{B}^+$,
$\ker(\theta)$ is principal and generated by 
$([\epsilon] - 1)/([\epsilon]^{1/p} - 1)$.
Also, $\bt{A}^+$ is complete for the $\ker(\theta)$-adic topology.
However, $\bt{B}^+$ is not (despite being $p$-adically complete); 
denote its completion by $\bB_{\dR}^+$.
This ring does not admit an extension of $\phi$, because $\phi$ is not
continuous for the $\ker(\theta)$-adic topology. Instead,
choose $\tilde{p} \in \bt{E}^+$ with $\bartheta(\tilde{p}) = p$,
form the $p$-adically completed polynomial ring $\bt{B}^+\langle x \rangle$,
and let $\bB^+_{\max} \subset \bB_{\dR}^+$ be the image of the
map $\bt{B}^+\langle x \rangle \to \bB_{\dR}^+$ under $x \mapsto [\tilde{p}]/p$
(that makes sense because $[\tilde{p}]/p - 1 \in \ker(\theta)$).
Then $\phi$ does extend to $\bB^+_{\max}$, and we can define
\[
\bt{B}_{\rig}^+ = \bigcap_{n \geq 0} \phi^n(\bB^+_{\max}).
\]
(Note for experts: we will substitute $\bB^+_{\max}$ for $\bB^+_{\crys}$, 
and this is okay because they give the same notion of admissibility.)
Put $\bB^+_{\st} = \bB^+_{\max}[\log [\tilde{p}]]$; this logarithm is only
a formal symbol, but there is a logical way to extend $\phi$ to it,
namely $\phi(\log[\tilde{p}]) = p \log [\tilde{p}]$.
To embed $\bB^+_{\st}$ into $\bB^+_{\dR}$, one must choose a branch of the
$p$-adic logarithm (i.e., a value of $\log(p)$),
and then map
\[
\log [\tilde{p}] \mapsto \log(p) - \sum_{i=1}^\infty \frac{(1-[\tilde{p}]/p)^{i}}{i}.
\]
One defines a \emph{monodromy operator} $N = d/d(\log [\tilde{p}])$ on
$\bB^+_{\st}$, satisfying
$N\phi = p \phi N$.

Finally, the non-plus variants $\bB_{\max}, \bB_{\st}, \bB_{\dR}$
are obtained from their plus counterparts by adjoining
$1/t$ for
\[
t = \log([\epsilon]) 
= \sum_{i=1}^\infty -\frac{(1-[\epsilon])^i}{i} \in \ker(\theta).
\]
Note that $\bB_{\dR}$ is a complete discrete valuation field with uniformizer $t$; it
is natural to equip it with the decreasing filtration
\[
\Fil^i \bB_{\dR} = t^i \bB_{\dR}^+.
\]
Also, the ring $\bB_e = \bB_{\max}^{\phi=1}$ sits in the so-called
\emph{fundamental exact sequence}:
\[
0 \to \QQ_p \to \bB_{\max}^{\phi=1} \to \bB_{\dR}/\bB_{\dR}^+ \to 0.
\]
This is loosely analogous to the exact sequence
\[
0 \to k \to k[u] \to k((t))/k \llbracket t \rrbracket \to 0
\]
for $u \in t^{-1} k \llbracket t \rrbracket \setminus k \llbracket t 
\rrbracket$.

\section{Permuting the steps}

We note in passing that it is possible to permute the steps of the
above constructions, with the aim of postponing the choice of the prime $p$.
This might help one present the theory so that the
infinite place becomes a valid choice of prime, under which one should recover
ordinary Hodge theory. So far, this is more speculation than reality; we report
here only the first steps, leaving future discussion to another occasion.

For starters, 
$\bt{A}^+$ can be constructed as the inverse limit of $W(\frako)$
under the Witt vector Frobenius map, and $\theta$ can be recovered as the composition
\[
\bt{A}^+ \to W(\frako) \to \frako
\]
of the first projection of the inverse limit with the residue map on Witt vectors.
(This observation was made by Lars Hesselholt and verified by Chris Davis.)
This still involves the use of $p$ in constructing $\frako = \frako_{\CC_p}$;
that can be postponed as follows.
Let $\barZZ$ be the ring of algebraic integers.
Let $R$ be the inverse limit of $W(\barZZ)$ under the Witt vector Frobenius;
then we get a map $\theta: R \to W(\barZZ) \to \barZZ$ by composing as above. 
Choose a prime $\frakp$ of $\barZZ$ above $p$ (thus determining a map
$\barQQ \to \barQQ_p$, up to an automorphism of $\barQQ_p$),
and put $(\tilde{p}) = \barZZ \cap p \barZZ_{\frakp}$; 
in other words, $(\tilde{p})$ are the algebraic integers
whose $\frakp$-adic valuation 
is at least that of $p$ itself. Then $\bt{A}^+$ is the 
completion of $R$ for the $\theta^{-1}(\tilde{p})$-adic topology.

We are still using $p$ via the definition of $W$ using $p$-typical Witt
vectors. One can
postpone the choice of $p$ further by using the big Witt vectors,
taking the inverse limit under all of the Frobenius maps, instead of the
$p$-typical Witt vectors. In this case, the
completion for the $\theta^{-1}(\tilde{p})$-adic topology splits into 
copies of $\bt{A}^+$ indexed by positive integers coprime to $p$, which
are shifted around by the prime-to-$p$ Frobenius maps.

\section{$B$-pairs as $p$-adic Hodge structures}

One of the principal points of departure for ordinary Hodge theory is
the comparison isomorphism
\[
H_{\Betti}^\cdot(X, \ZZ) \otimes_{\ZZ} \CC \cong H_{\dR}^\cdot(X, \CC).
\]
One then defines a \emph{Hodge structure} as a $\CC$-vector space
carrying the extra structures 
brought to the comparison isomorphism by the extra structures on
both sides: the integral structure on the left side, and the Hodge 
decomposition
on the right side.

The notion of a $B$-pair, recently introduced by Berger \cite{berger-bpairs},
performs an analogous function for the comparison isomorphism
\[
H^i_{\et}(X_{\barQQ_p}, \QQ_p) \otimes_{\QQ_p} \bB_{\dR}
\cong H^i_{\dR}(X, K) \otimes_K \bB_{\dR},
\]
where the extra structures are a Hodge filtration on the right side, 
and the compatible Galois actions on both sides.

A \emph{$B$-pair} over $K$ 
is a pair $(W_e, W_{\dR}^+)$, where $W_e$
is a finite free $\bB_e$-module equipped with a continuous semilinear
$G_K$-action, and $W_{\dR}^+$ is a finite
free $\bB_{\dR}^+$-submodule
of $W_{\dR} = W_e \otimes_{\bB_e} \bB_{\dR}$ stable under $G_K$ such that
$W_{\dR}^+ \otimes_{\bB_{\dR}^+} \bB_{\dR} \to W_{\dR}$ is an isomorphism.
(Note that $\bB_e$ is a B\'ezout domain,
i.e., an integral domain in which finitely generated ideals are
principal \cite[Proposition~2.2.3]{berger-bpairs}.)
The \emph{rank} of a $B$-pair $W$ is the common quantity
$\rank_{\bB_e} W_e = \rank_{\bB_{\dR}^+} W_{\dR}^+$.

\begin{lemma} \label{L:rank 1}
Every $B$-pair $W$ of rank $1$ over $K$ can be uniquely written as
$(\bB_e(\delta), t^i \bB_{\dR}^+ (\delta))$ 
for some character $\delta: G_K \to \QQ_p^\times$
and some $i \in \ZZ$.
\end{lemma}
\begin{proof}
See \cite[Lemme~3.1.4]{berger-bpairs}.
\end{proof}
We refer to the integer $i$ as the \emph{degree} of $W$. If $W$ has rank
$d>1$, we define the degree of $W$ as the degree of the determinant
$\det W = \wedge^d W$. We refer to the quotient $\mu(W) = (\deg W)/(\rank W)$
as the \emph{average 
slope} of $W$. (This is meant to bring to mind the theory of
vector bundles on curves; see the next section for an extension of this 
analogy.)

There is a functor
\[
V \mapsto W(V) = (V \otimes_{\QQ_p} \bB_e, V \otimes_{\QQ_p} \bB_{\dR}^+),
\]
from $p$-adic representations to $B$-pairs; it has the one-sided 
inverse
\[
W \mapsto V(W) = W_e \cap W_{\dR}^+,
\]
and so is fully faithful. However, not every $B$-pair corresponds to
a representation; for instance, in Lemma~\ref{L:rank 1}, only those
$B$-pairs with $i=0$ correspond to representations. 
We will return to the
structure of a general $B$-pair in the next two sections.

One can generalize many definitions and results from $p$-adic representations
to $B$-pairs. For instance, we set
\begin{align*}
D_{\max}(W) &= (W_e \otimes_{\bB_e} \bB_{\max})^{G_K} \\
D_{\st}(W) &= (W_e \otimes_{\bB_e} \bB_{\st})^{G_K} \\
D_{\dR}(W) &= (W_e \otimes_{\bB_e} \bB_{\dR})^{G_K};
\end{align*}
for $* \in \{\max, \st, \dR\}$, the map
\begin{equation} \label{eq:admissible2}
D_*(W) \otimes_{\bB_*^{G_K}} \bB_* \to V \otimes_{\QQ_p} \bB_*
\end{equation}
is injective, and we say that $W$ is \emph{$*$-admissible} 
if \eqref{eq:admissible2} is a bijection. In these terms, we
have the following extension of Theorem~\ref{T:pst}, again
due to Berger \cite{berger-inv}. (The closest analogue of this result
in ordinary Hodge theory is a theorem of Borel \cite[Lemma~4.5]{schmid}, 
that any polarized
variation of rational Hodge structures is quasi-unipotent.)

\begin{theorem} \label{T:pst2}
Let $W$ be a de Rham $B$-pair over $K$. Then there exists a finite
extension $L$ of $K$ such that the restriction of $W$ to $L$ is 
semistable.
\end{theorem}

Moreover, the $B$-pairs which are already crystalline over $K$ can be
described more explicitly; they form a category equivalent to
the category of filtered $\phi$-modules over $K$. (Such an object is
a finite dimensional $K_0$-vector space $V$, for $K_0$ the maximal unramified
extension of $K$, equipped with a semilinear
$\phi$-action and an exhaustive decreasing filtration
$\Fil^i V_K$ of $V_K = V \otimes_{K_0} K$.)

\section{Sloped representations}

For $h$ a positive integer and $a \in \ZZ$ coprime to $h$, 
we define an \emph{$(a/h)$-representation} as a finite-dimensional
$\QQ_{p^h}$-vector space $V_{a,h}$ equipped with a semilinear $G_K$-action
and a semilinear Frobenius action $\phi$ commuting with the $G_K$-action,
satisfying $\phi^h = p^a$. For instance, we may view a $p$-adic representation
as a 0-representation by taking $\phi = \id$.

We say that the $B$-pair $W$ is \emph{isoclinic of slope $a/h$} if
it occurs in the essential image of $W_{a,h}$; we say that $W$ is
\emph{\'etale} if it is isoclinic of slope 0.
\begin{theorem}
The functor
\[
V_{a,h} \mapsto W_{a,h}(V_{a,h}) = ((\bB_{\max} \otimes_{\QQ_{p^h}} V_{a,h})^{\phi=1},
\bB_{\dR}^+ \otimes_{\QQ_{p^h}} V_{a,h})
\]
from $(a/h)$-representations to $B$-pairs of slope $a/h$ is fully faithful.
(Here $\QQ_{p^h}$ denotes the unramified extension of $\QQ_p$ with
residue field $\FF_{p^h}$.)
\end{theorem}
\begin{proof}
See \cite[Th\'eor\`eme~4.3.3]{berger-bpairs}
 for the construction of a one-sided inverse.
\end{proof}

Using the equivalence of categories between $B$-pairs and 
$(\phi, \Gamma)$-modules (Theorem~\ref{T:equiv}), one deduces
the following properties of isoclinic $B$-pairs.

\begin{lemma} \label{L:maps}
\begin{enumerate}
\item[(a)]
If $0 \to W_1 \to W \to W_2 \to 0$ is exact and any two of
$W_1, W_2, W$ are isoclinic of slope $s$, then so is the third.
\item[(b)]
Let $W_1, W_2$ be isoclinic $B$-pairs of slopes $s_1 < s_2$.
Then $\Hom(W_1, W_2) = 0$.
\item[(c)]
Let $W_1, W_2$ be isoclinic $B$-pairs of slopes $s_1, s_2$.
Then $W_1 \otimes W_2$ is isoclinic of slope $s_1 + s_2$.
\end{enumerate}
\end{lemma}
\begin{proof}
See Proposition~1.5.6, Corollary~1.6.9, and Corollary~1.6.4, respectively,
of \cite{kedlaya-rel}.
\end{proof}

The following is a form of the author's slope filtration theorem
for Frobenius modules over the Robba ring.
\begin{theorem} \label{T:filtration}
Let $W$ be a $B$-pair. Then there is a unique filtration
$0 = W_0 \subset \cdots \subset W_l = W$ of $W$ by $B$-pairs
with the following properties.
\begin{enumerate}
\item[(a)]
For $i=1,\dots,l$, the quotient $W_i/W_{i-1}$ is a $B$-pair
which is isoclinic of slope $s_i$.
\item[(b)]
We have $s_1 < \cdots < s_l$.
\end{enumerate}
\end{theorem}
\begin{proof}
See \cite[Theorem~1.7.1]{kedlaya-rel}.
\end{proof}
A corollary is that
a $B$-pair $W$ is ``semistable'' (in the sense of vector bundles,
up to a reversal of the sign convention), meaning that
\[
0 \neq W_1 \subseteq W \implies \mu(W_1) \geq \mu(W),
\]
if and only if $W$ is isoclinic.

\begin{cor} \label{C:extend ineq}
The property of a $B$-pair having all slopes $\geq s$ (resp.\
$\leq s$) is stable under extensions.
\end{cor}

\section{Trianguline $B$-pairs}

One might reasonably ask why the study of $B$-pairs should be relevant
to problems only involving $p$-adic representations. One answer is that
there are many $p$-adic representations which become more decomposable
when viewed as $B$-pairs.

Specifically, following Colmez \cite{colmez},
we say that a $B$-pair $W$ is \emph{trianguline} (or
\emph{triangulable}) if $W$ admits a
filtration $0 = W_0 \subset \cdots \subset W_l = W$ in which each
quotient $W_i/W_{i-1}$ is a $B$-pair of rank 1. For instance,
by a theorem of Kisin \cite{kisin-overcon}, if $V$ is the 
two-dimensional representation corresponding to a classical
or overconvergent modular form of finite non-critical slope, then 
$W(V)$ is trianguline. 

As one might expect, the extra structure of a triangulation makes
trianguline $B$-pairs easier to classify.
For example, Colmez has classified all two-dimensional
trianguline $B$-pairs over $\QQ_p$, by calculating $\Ext(W_1, W_2)$ whenever
$W_1, W_2$ are $B$-pairs of rank 1.
He has also shown that their $\mathcal{L}$-invariants (in the sense of
Fontaine-Mazur) can be read off from the triangulation.

A study of the general theory of trianguline $B$-pairs has been initiated by
Bella\"\i che and Chenevier \cite{bellaiche-chenevier}, with the aim
of applying this theory to the study of Selmer groups associated to
Galois representations of dimension greater than 2 (e.g., those arising
from unitary groups). 
It is also hoped that this study will give insight into questions like
properness of the Coleman-Mazur eigencurve.
One feature apparent in the work of Bella\"\i che and
Chenevier, connected to
the results of the next section, is that the
trianguline property of a representation is reflected by the structure of
the corresponding deformation ring.
(A related notion is Pottharst's definition of a 
\emph{triangulordinary} representation \cite{pottharst}, 
which generalizes the notion of
an ordinary representation in a manner useful when considering duality
of Selmer groups.)

\section{Cohomology of $B$-quotients}

In this section, we describe a theorem of Liu \cite{liu} generalizing
Tate's fundamental results on the Galois cohomology of $p$-adic 
representations, and also the Ext group calculations of Colmez mentioned above.
However, to do this properly, we must work with a slightly
larger category than the $B$-pairs, because this category is 
not abelian: it contains kernels but not cokernels.

One can construct
a minimal abelian category containing the $B$-pairs as follows.
Define a \emph{$B$-quotient} as an inclusion $(W_1 \hookrightarrow W_2)$
of $B$-pairs. We put these in a category in which the morphisms
from $(W_1 \hookrightarrow W_2)$ to $(W'_1 \hookrightarrow W'_2)$ consist
of subobjects $X$ of $W_2 \oplus W'_2$ containing $W_1 \oplus W'_1$,
such that the composition $X \to W_2 \oplus W'_2 \to W_2$ is surjective
and the inverse image of $W_1$ is $W_1 \oplus W'_1$.
(One must also define addition and composition of morphisms, but the
reader should have no trouble reconstructing them.)
It can be shown that this yields an abelian category, into which
the $B$-pairs embed by mapping $W$ to $0 \hookrightarrow W$.

For 
$W = (W_1 \hookrightarrow W_2)$ a $B$-quotient, we define the
\emph{rank} of $W$ as $\rank(W_2) - \rank(W_1)$, and we say $W$
is \emph{torsion} if $\rank(W) = 0$.
We also write $\omega = V(\QQ_p(1))$.
One then has the following result, which in particular includes
the Euler characteristic formula and local duality theorems of Tate
\cite{milne}. (However, the proof uses these results as input, and so
does not rederive them independently.)
\begin{theorem}[Liu] \label{T:liu-bpairs}
Define the functor $H^0$ from $B$-quotients to $\QQ_p$-vector spaces
by $H^0(W) = \Hom(W_0, W)$, where $W_0$ is the trivial $B$-pair
$(\bB_e, \bB_{\dR}^+)$. Then $H^0$ extends to a universal
$\delta$-functor $(H^i)_{i=0}^\infty$
with the following properties.
\begin{enumerate}
\item[(a)]
For $W$ a $B$-quotient, $H^i(W)$ is finite dimensional over $\QQ_p$.
\item[(b)]
For $W$ a $B$-quotient, $H^i(W) = 0$ for $i > 2$.
\item[(c)]
For $W$ a torsion $B$-quotient, $H^2(W) =0$.
\item[(d)]
For $W$ a $B$-quotient, $\sum_{i=0}^2 (-1)^i \dim_{\QQ_p} H^i(W) = -[K:\QQ_p] 
\rank W$.
\item[(e)]
For $W$ a $B$-pair, the pairing 
\[
H^i(W) \times H^{2-i}(W^\dual \otimes \omega)
\to H^2(W \otimes W^\dual \otimes \omega)
\to H^2(\omega) \cong \QQ_p
\]
is perfect for $i=0,1,2$.
\end{enumerate}
Moreover, on the subcategory of $p$-adic representations, the $H^i$
are canonically naturally isomorphic to Galois cohomology (in a fashion
compatible with connecting homomorphisms).
\end{theorem}

\section{$(\phi, \Gamma)$-modules}

It would be somewhat misleading to leave the story at that, because it would
fail to give a real sense as to how one \emph{proves} any of the results
explained above. In fact, one tends to make proofs by working on
the opposite side of an equivalence of categories, in which the Galois
action is replaced by some more ``commutative'' data.
(As remarked in \cite{kisin-crys} in a slightly different context, these
objects bear some resemblance to the local versions of certain
geometric objects, called shtukas, appearing in Drinfel'd's approach
to the Langlands correspondence for function fields. How to profit from this
observation remains a mystery.)

The \emph{Robba ring} over a coefficient field $L$ is the ring of
formal Laurent series $\sum_{n \in \ZZ} c_n x^n$ with $c_n \in L$,
such that the series converges on some annulus $\epsilon < |x| < 1$
(depending on the series);
this is a B\'ezout domain.
Let $\calR = \bB^{\dagger}_{\rig, \QQ_p}$
be the Robba ring over $\QQ_p$ with the series variable $x$
identified with $\pi = [\epsilon] - 1$. 
This ring admits a Frobenius map $\phi$ given by $\phi(\pi) = (1+\pi)^p - 1$,
and an action of the group $\Gamma = \Gal(\QQ_p(\mu_{p^\infty})/\QQ_p)$
satisfying $\gamma(\pi) = (1+\pi)^{\chi(\gamma)} - 1$,
for $\chi: \Gamma \to \ZZ_p$ the cyclotomic character.

A \emph{$(\phi, \Gamma)$-module} over $\calR$ is a finite free
$\calR$-module $D$ equipped with an isomorphism $\phi^* D \to D$,
viewed as a semilinear action of $\phi$ on $D$, and a continuous
(for a topology which we won't describe here)
semilinear $\Gamma$-action commuting with $\phi$. If we only require
that $D$ be finitely presented, we call the result a
\emph{generalized $(\phi, \Gamma)$-module}. One then has the following
equivalence \cite[Th\'eor\`eme~3.2.7]{berger-bpairs}.

\begin{theorem}[Berger] \label{T:equiv}
The category of $B$-pairs over $\QQ_p$ is equivalent to the category of
$(\phi, \Gamma)$-modules over $\calR$.
\end{theorem}
There is an analogous equivalence over $K$, using a suitable finite
extension of $\calR$ which is again isomorphic to a Robba ring (over
a certain extension of $\QQ_p$); for simplicity, we omit further details.

A \emph{$(\phi, \Gamma)$-quotient} is a generalized
$(\phi, \Gamma)$-module of the form $D/D_1$ for some inclusion
$D_1 \hookrightarrow D$ of $(\phi, \Gamma)$-modules.
It is easily checked that any generalized $(\phi, \Gamma)$-module
occurring as a submodule or quotient of a 
$(\phi, \Gamma)$-quotient is also a $(\phi, \Gamma)$-quotient.
(We expect that not every generalized $(\phi, \Gamma)$-module
occurs as a $(\phi, \Gamma)$-quotient, but we have no example.)
\begin{cor}
The category of $B$-quotients is equivalent to
the category of $(\phi, \Gamma)$-quotients.
\end{cor}
In particular, any $B$-quotient $W$ admits a torsion subobject $X$
such that $W/X$ is a $B$-pair (since the same is true of 
finitely presented $\calR$-modules).

There is a sort of cohomology for generalized $(\phi, \Gamma)$-modules,
computed by a complex introduced by Colmez \cite{colmez}
derived from work of Herr \cite{herr}.
For $p > 2$, it can be constructed as follows.
Choose a topological generator $\gamma$ of $\Gamma$.
Given a generalized $(\phi, \Gamma)$-module $D$, for $i=0,1,2$,
let $H^i$ be the cohomology at position $i$ of the complex
\[
0 \to D \to D \oplus D \to D \to 0
\]
where the first map is $x \mapsto ((\gamma-1)x, (\phi-1)x)$
and the second map is
$(x,y) \mapsto ((\phi-1)x - (\gamma-1)y)$.
One also constructs cup product pairings, the only nonobvious one of which
is the map $H^1(D) \times H^1(D') \to H^2(D \otimes D')$ given by
$(x,y),(z,t) \mapsto y \otimes  \gamma(z) - x \otimes \phi(t)$.
(For $p = 2$, let $\gamma$ be a generator of $\Gamma/\{\pm 1\}$,
and replace $D$ by its invariants under $\{\pm 1\}$ in the construction
of the complex.)

We can now state what Liu actually proves in \cite{liu}. Note that
the more general result working with the Robba ring corresponding to $K$,
in which everything is the same except that the right side of (c) 
must be multiplied by $[K:\QQ_p]$, follows from the case $K =\QQ_p$
by a version of Shapiro's lemma for $(\phi, \Gamma)$-modules.
\begin{theorem}[Liu] \label{T:liu}
Let $D$ be a $(\phi, \Gamma)$-module over $\calR$.
\begin{enumerate}
\item[(a)] 
For $i=0,1,2$, $H^i(D)$ is finite dimensional over $\QQ_p$.
\item[(b)]
If $D$ is torsion, then $H^2(D) = 0$.
\item[(c)]
We have $\sum_{i=0}^2 (-1)^i \dim_{\QQ_p} H^i(D) = - \rank(D/D_{\tors})$.
\item[(d)]
If $D$ is free, then for 
$i=0,1,2$, the pairing
\[
H^i(D) \times H^{2-i}(D^\dual \otimes \omega)
\to H^2(D \otimes D^\dual \otimes \omega) \to H^2(\omega) \cong \QQ_p
\]
is perfect.
\end{enumerate}
Moreover, on the subcategory of $p$-adic representations, the $H^i$
are canonically naturally isomorphic to Galois cohomology (in a fashion
compatible with connecting homomorphisms).
\end{theorem}

\appendix
\section{Cohomology of $B$-pairs}

For $W$ a $B$-pair or $B$-quotient, 
let $D(W)$ be the associated $(\phi, \Gamma)$-module or
$(\phi, \Gamma)$-quotient 
(see \cite[Th\'eor\`eme~2.2.7]{berger-bpairs} for the construction),
and write $H^i(W)$ instead of $H^i(D(W))$.
For this to be consistent with Theorem~\ref{T:liu-bpairs},
we must have $H^0(D(W)) = \Hom(W_0,W)$; fortunately,
Theorem~\ref{T:equiv} implies $\Hom(W_0,W) = \Hom(D(W_0), D(W))$
and it is trivial to check that $\Hom(D(W_0),D(W)) = H^0(D(W))$.
All of the
assertions of Theorem~\ref{T:liu-bpairs} follow from Theorem~\ref{T:liu}
as stated, except for the fact that the $\delta$-functor formed by the $H^i$
is universal.

The purpose of this appendix is to fill in this gap, as the proof does not
appear elsewhere. It is mostly meant for experts,
and does not maintain the expository style we have attempted to maintain
in the main text. 

Besides Theorem~\ref{T:liu}, and the equivalence given by
Theorem~\ref{T:equiv},
we will make repeated use of the following facts.
\begin{itemize}
\item
If $0 \to W_1 \to W \to W_2 \to 0$ is an exact sequence of $B$-quotients,
and $W_1, W_2$ are $B$-pairs, then so is $W$.
\item
If $D \to D_2 \to 0$ is an exact sequence of generalized
$(\phi, \Gamma)$-modules, and $D$ is a $(\phi, \Gamma)$-quotient, then so is
$D_2$.
\end{itemize}

\begin{lemma} \label{L:twist}
Let $W$ be a $B$-pair. Then there exists a character $\delta:
G_K \to \QQ_p^\times$ such that $\Hom(W(\delta), W) = 0$.
\end{lemma}
\begin{proof}
We will show that all but finitely many $\delta$ have this property,
by induction on $\rank W$.
Suppose first that $\rank W = 1$; by Lemma~\ref{L:rank 1},
we can write $W = t^i W_0 \otimes W(\delta')$ for
some $i \in \ZZ$ and some character $\delta': G_K \to \QQ_p^\times$.
Then to have a nonzero map $W_e(\delta) \to W_e = W_e(\delta')$,
we must have $\delta = \delta'$.

Suppose now that $\rank W > 1$. If there is no $\delta_0$ such that
$\Hom(W(\delta_0),W) \neq 0$, we are done. Otherwise, we have a short exact
sequence
\[
0 \to t^i W(\delta_0) \to W \to W_1 \to 0
\]
for some $i \in \ZZ$. By the induction hypothesis,
for all but finitely many $\delta$,
$\Hom(W(\delta), t^i W(\delta_0)) = \Hom(W(\delta), W_1) = 0$,
so $\Hom(W(\delta), W) = 0$.
\end{proof}

By the \emph{slopes} of a $B$-pair $W$, we will mean those numbers $s_1,
\dots, s_l$ occurring in Theorem~\ref{T:filtration}.
\begin{lemma} \label{L:extend positive}
Let $0 \to W_1 \to W \to W_2 \to 0$ be a nonsplit exact sequence of $B$-pairs,
in which $W_1$ has positive slopes, and $W_2$ has 
rank $1$ and degree $-1$. Then $W$ has nonnegative slopes.
\end{lemma}
\begin{proof}
Suppose the contrary; then 
by Theorem~\ref{T:filtration},
there is a short exact sequence
$0 \to X \to W \to W/X \to 0$ of $B$-pairs
with $X$ isoclinic of negative slope.
Note that $\deg(X) = \deg(X \cap W_1) + \deg(X/(X \cap W_1))$;
by Lemma~\ref{L:maps}, the first term is positive unless $X \cap W_1 = 0$,
and the second 
term is at least $-1$. The only way to have $\deg(X) < 0$ is to
have $X \cap W_1 = 0$ and $X \to W_2$ an isomorphism, but this
only happens if the sequence splits.
\end{proof}

\begin{lemma} \label{L:neg slopes}
If $W$ is a $B$-pair with negative slopes, then $H^2(W) = 0$.
\end{lemma}
\begin{proof}
By Theorem~\ref{T:liu}, $H^2(W)$ is dual to $H^0(W^\dual \otimes \omega)$,
which vanishes by Lemma~\ref{L:maps}.
\end{proof}

\begin{prop} \label{P:extend etale}
Let $W$ be a $B$-pair whose slopes are all 
nonnegative. Then there exists a short exact sequence $0 \to W \to X \to Y 
\to 0$ such that $X$ is \'etale.
\end{prop}
\begin{proof}
(This argument is a variant of the argument used in \cite{liu}
to reduce Tate duality to the \'etale case.)
We induct on $s = \deg(W) \geq 0$. 
If $s=0$, then $W$ is \'etale and
we are done. Otherwise, let $W_1$ be the maximal \'etale subobject of $W$
(which may be zero) and put $W_2 = W/W_1$. By applying
Lemma~\ref{L:twist}, we can choose a $B$-pair $T$ which is isoclinic
of rank 1 and degree -1, such that
$H^0(W_1^\dual \otimes T \otimes \omega) = 0$.

Put $W' = W \otimes T^\dual, W'_1 = W_1 \otimes T^\dual,
W'_2 = W_2 \otimes T^\dual$. Apply Theorem~\ref{T:liu} to obtain
\begin{align*}
\dim H^1(W') &= \dim H^0(W') + \dim H^2(W')
+ [K:\QQ_p] \rank(W') \\
\dim H^1(W'_1) &= \dim H^0(W'_1) + \dim H^2(W'_1)
+ [K:\QQ_p] \rank(W'_1).
\end{align*}
We claim each of the three 
terms in the first row is greater than or equal to the corresponding term
in the second row. The inequality is an equality on the first terms,
because both terms vanish by Lemma~\ref{L:maps}.
The inequality on the second terms holds because
by Theorem~\ref{T:liu}, $H^2(W'_1) = 
H^0((W'_1)^\dual \otimes \omega)^\dual = 0$. 
The inequality on the third terms 
is strict because $W \neq W_1$ by assumption.

We conclude that $\dim H^1(W') > \dim H^1(W'_1)$, so in particular
the map $H^1(W'_1) \to H^1(W')$ cannot be surjective. Choose a class in
$H^1(W')$ not in the image of this map; then the resulting class in
$H^1(W'_2)$ is nonzero. Form the corresponding extension
\[
0 \to W \to X \to T \to 0;
\]
then we also have an exact sequence
\[
0 \to W_1 \to X \to X/W_1 \to 0
\]
and a nonsplit exact sequence
\[
0 \to W_2 \to X/W_1 \to T \to 0.
\]
By Lemma~\ref{L:extend positive}, $X/W_1$ has nonnegative slopes,
as then does $X$ by Corollary~\ref{C:extend ineq}.
Since $\deg(X) = s - 1$, we may deduce the claim by the induction
hypothesis.
\end{proof}
\begin{cor} \label{C:extend etale}
Let $W$ be a $B$-pair.
Then for $n$ sufficiently large, there exists a 
short exact sequence $0 \to W \to X \to Y 
\to 0$ such that $X$ is isoclinic of slope $-n$.
\end{cor}
\begin{proof}
Apply Proposition~\ref{P:extend etale} to $W \otimes T$, for $T$
of rank 1 and degree $n$.
\end{proof}

\begin{lemma} \label{L:coverable isoclinic}
Let $W$ be a $B$-quotient. Then for $n$ sufficiently large, 
we can write $W = (X_1 \hookrightarrow X_2)$ 
with $X_1$ isoclinic of slope $-n$.
\end{lemma}
\begin{proof}
Put $W = (W_1 \hookrightarrow W_2)$.
By Corollary~\ref{C:extend etale}, for $n$ large,
we can find a short exact sequence $0 \to W_1 \to X_1 \to Y \to 0$
of $B$-pairs with $X_1$ isoclinic of slope $-n$.
By forming a pushout, we obtain a diagram
\[
\xymatrix{
0 \ar[r] & W_1 \ar[d] \ar[r] & W_2 \ar[d] \ar[r] & W
\ar[d] \ar[r] & 0 \\
0 \ar[r] & X_1 \ar[r] & X_2 \ar[r] & W \ar[r] & 0
}
\]
of $B$-quotients with $0 \to W_2 \to X_2 \to Y \to 0$ exact.
In particular, $X_2$ is a $B$-pair, not just a $B$-quotient,
and $W \cong (X_1 \hookrightarrow X_2)$.
\end{proof}

\begin{prop} \label{P:extend coverable}
Let $D_1, D_2$ be $(\phi, \Gamma)$-quotients. Then the group $\Ext(D_2, D_1)$
is the same whether computed in the category of $(\phi, \Gamma)$-quotients
or in the category of generalized $(\phi, \Gamma)$-modules.
Consequently, the former group is equal to $H^1(D_2^\dual \otimes D_1)$.
\end{prop}
\begin{proof}
 Let
$0 \to D_1 \to D \to D_2 \to 0$ be an extension of generalized
$(\phi, \Gamma)$-modules.
First suppose $D_2$ is actually a $(\phi, \Gamma)$-module.
By Lemma~\ref{L:coverable isoclinic}, for $n$ large,
we can write $D_1$ as a quotient $E_2/E_1$ of $(\phi, \Gamma)$-modules
with $E_1$ isoclinic of slope $-n$. For $n$ large,
we have $H^2(D_2^\dual \otimes E_1) = 0$ by
Lemma~\ref{L:neg slopes}.
Since
\[
H^1(D_2^\dual \otimes E_2) \to H^1(D_2^\dual \otimes D_1)
\to H^2(D_2^\dual \otimes E_1) = 0
\]
is exact, we can lift the original exact sequence to a sequence
$0 \to E_2 \to F \to D_2 \to 0$. 
Since $E_2$ and $D_2$ are $(\phi, \Gamma)$-modules, so is $F$;
since  $F \to D$ is surjective, $D$ is a $(\phi, \Gamma)$-quotient.

In the general case, write $D_2$ as a quotient $E_2/E_1$ of
$(\phi, \Gamma)$-modules.
By forming a pullback, we obtain a diagram
\[
\xymatrix{
0 \ar[r] & D_1 \ar[d] \ar[r] & F
\ar[d] \ar[r] & E_2 \ar[d] \ar[r] & 0 \\
0 \ar[r] & D_1 \ar[r] & D \ar[r] & D_2 \ar[r] & 0
}
\]
of generalized $(\phi, \Gamma)$-modules, with $F \to D$ surjective.
By the previous paragraph, $F$ is a $(\phi, \Gamma)$-quotient, as then is $D$.
\end{proof}

\begin{prop} \label{P:efface h2}
For any $B$-quotient $W$, there exists an injection $W \hookrightarrow X$
with $H^2(X) = 0$.
\end{prop}
\begin{proof}
Let $T$ be the maximal torsion subobject of $W$.
By Theorem~\ref{T:liu}, $H^2(T) = 0$.
By Corollary~\ref{C:extend etale}, for $n > 0$ large,
we can construct an extension 
$0 \to W/T \to X_0 \to X_1 \to 0$ of $B$-pairs with
$X_0$ isoclinic of slope $-n$. By Lemma~\ref{L:neg slopes},
$H^2(X_0) = 0$.

In the exact sequence
\[
H^1(X_0^\dual \otimes T) \to H^1((W/T)^\dual \otimes T) \to H^2(X_1^\dual 
\otimes T),
\]
the last term vanishes by Theorem~\ref{T:liu}
because $X_1^\dual \otimes T$ is torsion.
Hence we can lift the extension $0 \to T \to W \to W/T \to 0$
to an extension $0 \to T \to X \to X_0 \to 0$.
(For this we need Proposition~\ref{P:extend coverable}, to
assert that the extension of generalized $(\phi, \Gamma)$-modules
actually arises from an an extension of
$B$-quotients.)
Now $W$ injects into $X$,
and $0 = H^2(T) \to H^2(X) \to H^2(X_0) = 0$ is exact, so $H^2(X) = 0$
as desired.
\end{proof}

\begin{proof}[Proof of Theorem~\ref{T:liu-bpairs}]
The functor $H^1$ is weakly effaceable because 
it coincides with the Ext group by
Proposition~\ref{P:extend coverable}.
(Namely, the class in $H^1(W)$ corresponding to an extension
$0 \to W \to X \to W_0 \to 0$ always vanishes in $H^1(X)$.)
The functor $H^2$ is effaceable by Proposition~\ref{P:efface h2}.
As in \cite[\S 2]{weibel}, the $H^i$ thus form a universal $\delta$-functor;
as noted above, this plus
Theorem~\ref{T:liu} proves everything.
\end{proof}

One can prove some variant statements by the same argument.
For one, the $H^i$ form a universal
$\delta$-functor on the full category of
generalized $(\phi, \Gamma)$-modules. For another,
the $H^i$ can be computed by taking derived functors in the ind-category of
generalized $(\phi, \Gamma)$-modules, i.e., the category of direct limits
of generalized $(\phi, \Gamma)$-modules.

\section*{Acknowledgments}

Thanks to the organizers of the 2007 Journ\'ees Arithm\'etiques
for the invitation to present this material there,
to Jay Pottharst for reviewing an earlier draft,
to Peter Schneider for suggesting the 
formulation of Theorem~\ref{T:liu-bpairs} in terms of satellite
functors, and to Laurent Berger and Ruochuan Liu for helpful
discussions. The author was supported by
NSF CAREER grant DMS-0545904 and a Sloan Research Fellowship.

\end{document}